\documentclass[twoside]{article}
\usepackage[a4paper]{geometry}
\usepackage[latin1]{inputenc} 
\usepackage[T1]{fontenc} 
\usepackage{RR}
\usepackage{hyperref}

\usepackage{amssymb}
\usepackage{amsmath}
\usepackage{amsthm}
\usepackage{algorithm,algorithmic}
\usepackage{graphicx}

\newtheorem{prop}{Proposition}
\newtheorem{lemma}{Lemma}

\newcommand{\Tr}{\textrm{Tr}}

\usepackage{comment}

\RRNo{8078}
\RRdate{October 2012}
\RRversion{3}
 \RRdater{June 2014}
\RRauthor{
Mahmoud El Chamie\thanks{Inria Sophia Antipolis, France, Mahmoud.El\_Chamie@inria.fr}\thanks{University of Nice Sophia Antipolis, France}%
  \and
Giovanni Neglia\thanks{Inria Sophia Antipolis, France, Giovanni.Neglia@inria.fr}%
\and 
Konstantin Avrachenkov\thanks{Inria Sophia Antipolis, France, K.Avrachenkov@sophia.inria.fr} }
\authorhead{M.El Chamie \& G.Neglia \& K.Avrachenkov}
\RRtitle{Une proc\'edure distribu\'ee de s\'election de poids dans les protocoles de consensus par minimisation de la norme de Schatten}
\RRetitle{Distributed Weight Selection in Consensus Protocols by Schatten Norm Minimization}
\titlehead{Distributed Weight Selection in Consensus Protocols by Schatten Norm Minimization}

\RRresume{Dans les protocoles de consensus, les n{\oe}uds d'un r\'eseau calculent it\'erativement une moyenne pond\'er\'ee de leurs mesures et celles de leurs voisins. Le protocole converge vers la moyenne des mesures initiales  de tous les n{\oe}uds pr\'esents dans le r\'eseau. La vitesse de convergence des protocoles de consensus d\'epend des poids s\'electionn\'es sur les liens entre voisins. Nous abordons dans cet article la question suivante : comment choisir les poids dans un r\'eseau donn\'e afin d'avoir une plus grande vitesse de convergence du protocole? Nous approchons le probl\`eme de la s\'election optimale de poids avec un probl\`eme de minimisation de la $p$-norme de Schatten.  Ce dernier est r\'esolu de mani\`ere totalement distribu\'ee gr\^ace \`a une m\`ethode du gradient. Selon la valeur du param\`etre  $p$, nous pouvons trouver un compromis entre la qualit\'e de la solution (c'est-\`a-dire la vitesse de convergence) et les co\^ut en termes de communication et calcul (e.g.~nombre de messages \'echang\'es et volume de donn\'ees trait\'ees). Les r\'esultats des simulations montrent que notre approche fournit une tr\`es bonne performance m\^eme avec un \'echange d'informations limit\'e. La proc\'edure d'optimisation des poids peut \'egalement se d\'erouler en simultan\'e avec le protocole de consensus.
}
\RRabstract{In average consensus protocols, nodes in a network perform an iterative weighted average of their estimates and those of their neighbors. The protocol converges to the average of initial estimates of all nodes found in the network. The speed of convergence of average consensus protocols depends on the weights selected on links (to neighbors). We address in this report how to select the weights in a given network in order to have a fast speed of convergence for these protocols. We approximate the problem of optimal weight selection by the minimization of the Schatten $p$-norm of a matrix with some constraints related to the connectivity of the underlying network. We then provide a totally distributed gradient method to solve the Schatten norm optimization problem. By tuning the parameter $p$ in our proposed minimization, we can simply trade-off the quality of the solution (i.e. the speed of convergence) for communication/computation requirements (in terms of number of messages exchanged and volume of data processed).  Simulation results show that our approach provides very good performance already for values of $p$ that only needs limited information exchange. The weight optimization iterative procedure can also run in parallel with the consensus protocol and form a joint consensus--optimization procedure. }
\RRmotcle{consensus de moyenne, s\'election de poids, la norme de Schatten, algorithmes distribu\'es de gradient }
\RRkeyword{average consensus, weight selection algorithms, Schatten norm minimization, distributed gradient algorithms}
\RRprojets{Maestro}
\RCSophia 

\begin{document}
\makeRR   
\tableofcontents
\newpage

\section{Introduction}\label{sec:intro}
A network is formed of nodes (or agents) and communication links that allow these nodes to share information and resources. Algorithms for efficient routing and efficient use of resources are proposed to save energy and speed up the processing. For small networks, it is possible for a central unit to be aware of all the components of the network and decide how to optimally use a resource on a global view basis. As networks expand, the central unit needs to handle a larger amount of data, and centralized optimization may become unfeasible especially when the network is dynamic. In fact, the optimal configuration needs to be computed whenever a link fails or there is any change in the network. Moreover, nodes may have some processing capabilities that are not used in the centralized optimization. With these points in mind, it becomes more convenient to perform distributed optimization relying on local computation at each node and local information exchange between neighbors. Such distributed approach is intrinsically able to adapt to local network changes. 
A significant amount of research on distributed optimization in networks has recently been carried out. New faster techniques  (\cite{Wei11a,Wei11b,Ghadimi11}) have been proposed for the traditional dual decomposition approach for separable problems that is well known in the network community since Kelly's seminal work on TCP (\cite{Kelly98}). A completely different approach has been recently proposed in~\cite{Nedic09}: it combines a consensus protocol, that is used to distribute the computations among the agents, and a subgradient method for the minimization of a local objective. Convergence results hold in the presence of constraints (\cite{Nedic10}), errors due to quantization (\cite{Nedic09b}) or to some stochastic noise (\cite{Ram10}) and in dynamic settings (\cite{Nedic09a,Lobel11,Masiero11}). 
Finally a third approach relies on some intelligent random exploration of the possible solution space, e.g.~using genetic algorithms (\cite{Alouf10}) or the annealed Gibbs sampler (\cite{Kauffmann07}).

In this report we study distributed techniques to optimally select the weights of average consensus protocols (also referred to as ave-consensus protocols or algorithms). These protocols allow nodes in a network, each having a certain measurement or state value, to calculate the average across all the values in the network by communicating only with their neighbors. Consensus algorithms are used in various applications such as environmental monitoring of wireless sensor networks and cooperative control of a team of machines working together to accomplish some predefined goal. For example, a group of vehicles moving in formation to the same target must reach consensus on the speed and direction of their motion to prevent collisions. Although the average seems to be a simple function to compute, it is an essential element for performing much more complex tasks including optimization, source localization, compression, and subspace tracking (\cite{Rabideau:1996, Ren:2007, Boyd06}). 

In the ave-consensus protocol, each node first selects weights for each of its neighbors,  then at each iteration the estimates are exchanged between neighbors and each node updates its own estimate by performing a weighted sum of the values received (\cite{Saber07, Ren:2005}). Under quite general conditions the estimates asymptotically converge to the average across all the original values. 
The weights play an essential role to determine the speed of convergence of the ave-consensus. 
For this reason in this report we study how to select the weights in a given network in order to have fast convergence independently from initial nodes' estimates.
In \cite{Xiao04}, the authors refer to this problem as the Fastest Distributed Linear Averaging (FDLA) weight selection problem. They show FDLA problem is equivalent to maximize the spectral gap of the weight matrix $W$ and that, with the additional requirement of $W$ being symmetric, is a convex \emph{non-smooth}  optimization problem. Then, the (symmetric) FDLA problem  can be solved offline by a centralized node using interior point methods, but, as we discussed above, this approach may not be convenient for large scale networks and/or when the topology changes over time. 

In this work, we propose to select the consensus weights as the values that minimize the Schatten $p$-norm of the weight matrix $W$ under some constraints (due also to the network topology). The Schatten $p$-norm of a matrix is the $p$-norm of its singular values as we will see later. We show that this new optimization problem can be considered an approximation of the original problem in \cite{Xiao04} (the FDLA) and we reformulate our problem as an equivalent unconstrained, convex and \emph{smooth} minimization that can be solved by the gradient method. More importantly, we show that in this case the gradient method can be efficiently distributed among the nodes.  
We describe a distributed gradient procedure to minimize the Schatten $p$-norm for an even integer $p$ that requires each node to recover information from nodes that are up to $\frac{p}{2}$-hop distant. Then the order $p$ of the Schatten norm is a tuning parameter that allows us to trade off the quality of the solution for the amount of communication/computation needed. In fact the larger $p$ the more precise the approximation, but also the larger the amount of information nodes need to exchange and process. The simulations are done on real networks (such as Enron's internal email graph and the dolphins social network) and on random networks (such as Erdos Renyi and Random Geometric Graphs). Our simulation results show that our algorithm provides very good performance in comparison to other distributed weighted selection algorithms already for $p=2$, i.e. when each node needs to collect information only from its direct neighbors.
Finally, we show  that nodes do not need to run our weight optimization algorithm \emph{before} being able to start the consensus protocol to calculate the average value, but the two can run in parallel.

The report is organized as follows:
In Section~\ref{PF} we formulate the problem we are considering and we give the notation used across the report.
Section~\ref{relatedwork} presents the related work for the weight selection problem for average consensus.
In Section \ref{TM} we propose Schatten $p$-norm minimization as an approximation of the original problem and in section~\ref{COTM} we show how its solution can be computed in a distributed way and evaluate its computation and communication costs.
 Section \ref{PE} compares the performance of our algorithm and that of other known weight selection algorithms on different graph topologies (real and random graphs). We also investigate the case when the weight optimization algorithm and the consensus protocol runs simultaneously, and then the weight matrix changes at every time slot. Section~\ref{EXT} discusses some methods to deal with potential instability problems and with misbehaving nodes.
Section \ref{Conc} summarizes the report.

\section{Problem Formulation} \label{PF}

Consider a network of $n$ nodes that can exchange messages between each other through communication links. Every node in this network has a certain value (e.g.~a measurement of temperature in a sensor network or a target speed in a unmanned vehicle team), and each of them calculate the average of these values through distributed linear iterations. The network of nodes can be modeled as a graph $G=(V,E)$ where $V$ is the set of vertices, labeled from $1$ to $n$, and $E$ is the set of edges, then $(i,j)\in E$ if nodes $i$ and $j$ are connected and can communicate (they are neighbors) and $|E|=m$. We label the edges from $1$ to $m$. If link $(i,j)$ has label $l$, we write $l \sim (i,j)$.  Let also $N_i$ be the neighborhood set of node $i$. All graphs in this report are considered to be \emph{connected} and \emph{undirected}. 
Let $x_i(0)\in \mathbb{R}$ be the initial value at node $i$. We are interested in computing the average 
$$x_{ave}=(1/n)\sum _{i=1}^nx_i(0),$$ 
in a decentralized manner with nodes only communicating with their neighbors. 
The network is supposed to operate synchronously: when a global clock ticks, all nodes in the system perform the iteration of the averaging protocol. At iteration $k+1$, node $i$ updates its state value $x_i$ as follows:
\begin{equation}\label{wsum}
x_i(k+1)= w_{ii}x_i(k)+\sum _{j\in N_i}w_{ij}x_j(k),
\end{equation}
where $w_{ij}$ is the weight selected by node $i$ for the value sent by its neighbor $j$ and $w_{ii}$ is the weight selected by node $i$ for it own value. 
As it is commonly assumed, in this report we consider that two neighbors select the same weight for each other, i.e.~$w_{ij}=w_{ji}$. 
The matrix form equation is:
\begin{equation}\label{Mwsum}
\mathbf{x}(k+1)=W\mathbf{x}(k),
\end{equation}
where $\mathbf{x}(k)$ is the state vector of the system and $W$ is the weight matrix. The main problem we are considering in this report is how a node $i$ can choose the weights $w_{ij}$ for its neighbors so that the state vector $\mathbf{x}$ of the system converges fast to consensus. There are centralized and distributed algorithms for the selection of $W$, but in order to explain them, we need first to provide some more notation. 
We denote  by  $\mathbf{w}$ the vector of dimensions $m\times 1$, whose $l$-th element $w_l$ is the weight associated to link $l$, then if $l\sim(i,j)$ it holds $w_l=w_{ij}=w_{ji}$. $A$ is the adjacency matrix of graph $G$, i.e. $a_{ij}=1$ if $(i,j) \in E$ and $a_{ij}=0$ otherwise. $\mathcal{C}_G$ is the set of all real $n\times n$ matrices $M$ following graph $G$, i.e. $m_{ij}=0$ if $(i,j) \notin E$. $D$ is a diagonal matrix where $d_{ii}$ (or simply $d_i$) is the degree of node $i$ in the graph $G$. $\mathcal{I}$ is the $n\times m$ incidence matrix of the graph, such that for each $l\sim(i,j)\in E\;$ $\mathcal{I}_{il}=+1$ and $\mathcal{I}_{jl}=-1$ and the rest of the elements of the matrix are null. $L$ is the laplacian matrix of the graph, so $L=D-A$. It can also be seen that $L=\mathcal{I}\mathcal{I}^T$. 
The $n\times n$ identity matrix is denoted by $I$. 
Given that $W$ is real and symmetric, it has real eigenvalues (and then they can be ordered). We denote by $\lambda _i$ the $i$-th largest eigenvalue of $W$, and by $\mu$ the largest eigenvalue in module non considering $\lambda_1$, i.e.~$\mu=\max\{\lambda_2,-\lambda_n\}$. $\sigma _i$ is the $i$-th largest singular value of a matrix. $\Tr(X)$ is the trace of the matrix $X$ and $\rho(X)$ is its spectral radius.  $||X||_{\sigma p}$ denotes the Schatten $p$-norm of matrix $X$, i.e. $||X||_{\sigma p}=(\sum _i\sigma _i^p)^{1/p}$. Finally we use the symbol $\frac{\text{d}\hfill }{\text{d}X}f(X)$, where $f$ is a differentiable scalar-valued function $f(X)$ with matrix argument $X\in \mathbb{R}^{m\times n}$, to denote the $n\times m$ matrix whose $(i,j)$ entry is $\frac{\partial f(X)}{\partial x_{ji}}$. 
Table \ref{tab:notation} summarizes the notation used in this report.

\begin{table}[h]
\caption{Notion } 
\centering
\scalebox{1}{
\begin{tabular}{lll}
\hline\hline
Symbol & Description & Dimension\\ 
\hline
$G$ & network of nodes and links & -\\
$V$ & set of nodes/vertices & $|V|=n$\\
$E$ & set of links/edges & $|E|=m$\\
$\mathbf{x}(k)$ & state vector of the system at iteration $k$ & $n\times 1$\\
$W$ & weight matrix & $n\times n$\\
$\mathbf{W}_i$ & vector of weights selected by node $i$ to its neighbors & $d_i\times 1$\\
$\mathbf{w}$ & vector of weights on links & $m\times 1$\\
$\text{diag}(\mathbf{v})$ & diagonal matrix having the elements of the $n\times 1$ vector $\mathbf{v}$ & $n\times n$ \\
$\mathcal{C}_G$ & set of $n\times n$ real matrices following $G$& - \\
$D$ & degree diagonal matrix & $n\times n$\\
$A$ & adjacency matrix of a graph & $n\times n$\\
$\mathcal{I}$ & incidence matrix of a graph & $n\times m$\\
$L$ & laplacian matrix $L=D-A=\mathcal{I}\mathcal{I}^T$ & $n\times n$\\
$\lambda _i$ & $i$th largest eigenvalue of $W$ & scalar\\
$\Lambda$ & eigenvalue diagonal matrix $\Lambda _{ii}=\lambda _i$ & $n\times n$\\ 
$\sigma _i$ & $i$th largest singular value & scalar\\
$\mu$ & second largest eigenvalue in magnitude of $W$ & scalar\\
$\rho (X)$ & spectral radius of matrix $X$ & scalar\\
$\Tr(X)$ & trace of the matrix $X$ & scalar\\
$||X||_{\sigma p}$ & Schatten $p$-norm of a matrix $X$ & scalar\\
$\frac{\text{d}\hfill}{\text{d}X}f(X)$ & Derivative of $f(X)$, $X\in \mathbb{R}^{m\times n}$, $f(X)\in \mathbb{R}$ & $n\times m$\\
$P_S(.)$ & Projection on a set $S\subset\mathbb{R}^m$ & $m\times 1$\\
\hline
\end{tabular}}
\label{tab:notation}
\end{table}

\subsection{Convergence Conditions}
In \cite{Xiao04} the following set of conditions is proven to be necessary and sufficient to guarantee convergence to consensus for any initial condition:
\begin{eqnarray}
& & \mathbf{1}^T W = \mathbf{1}^T, \label{e:stoch}\\
& & W \mathbf{1}=\mathbf{1},\label{e:stoch2}\\
& & \rho(W - \frac{1}{n}\mathbf{1}\mathbf{1}^T)<1, \label{e:spectral}
\end{eqnarray}
where $\mathbf{1}$ is the vector of all ones. We observe that the weights are not required to be non-negative.
Since we consider $W$ to be symmetric in this report, then the first two conditions are equivalent to each other and equivalent to the possibility to write the weight matrix as follows: $W=I-\mathcal{I}\times \text{diag}(\mathbf{w})\times \mathcal{I}^T$, where $I$ is the identity matrix and $\mathbf{w}\in \mathbb{R}^m$ is the vector of all the weights on links  $w_l$, $l=1...m$.

\subsection{Fastest Consensus}
\label{s:fastest}
The speed of convergence of the system given in \eqref{Mwsum} is governed by how fast $W^k$ converges. Since $W$ is real symmetric, it has real eigenvalues and it is diagonalizable. We can write $W^k$ as follows (\cite{Meyer:2000}):
\begin{equation}\label{Wprojdec}
W^k=\sum _i\lambda _i^k G_i,
\end{equation}
where the matrices $G_i$'s have the following properties: $G_i$ is the projector onto the null-space of $W-\lambda_i I$ along the range of $W-\lambda_i I$, $\sum_i G_i=I$ and $G_i G_j=0^{n\times n} \ \ \forall i\neq j$. 
Conditions~(\ref{e:stoch}) and~(\ref{e:spectral}) imply that $1$ is the largest eigenvalue of $W$ in module and is simple. Then $\lambda_1=1$,  $G_1=1/n \mathbf{1}\mathbf{1}^T$ and $|\lambda_i|<1$ for $i>1$. 
From the above representation of $W^k$, we can deduce two important facts:
\begin{enumerate}
\item First we can check that $W^k$ actually converges, in fact we have $\lim _{k\rightarrow \infty}\mathbf{x}(k)=$ $\lim_{k\rightarrow \infty} W^k \mathbf{x}(0)=\frac{1}{n}\mathbf{1}\mathbf{1}^T x(0)= x_{ave}\mathbf{1}$ as expected.
\item Second, the speed of convergence of $W^k$ is governed by the second largest eigenvalue in module, i.e.~on $\mu = \max \{\lambda _2,-\lambda _n\}=\rho\left(W-G_1 \right)$. For obtaining the fastest convergence, nodes have to select weights that minimizes $\mu$, or equivalently maximize the spectral gap\footnote{
	The spectral gap is the difference between the largest eigenvalue in module and the second largest one in module. In this case it is equal to $1-\mu$.
} of $W$.
\end{enumerate}
Then the problem of finding the weight matrix that guarantees the fastest convergence can be formalized as follows:

\begin{equation}\label{minim}
\begin{aligned}
& \underset{W}{\text{Argmin}}
& & \mu (W)\\
& \text{subject to}
& & W=W^T, \\
&&& W\mathbf{1}=\mathbf{1},\\
&&& W\in \mathcal{C} _G,
\end{aligned}
\end{equation}
where the last constraint on the matrix $W$ derives from the assumption that nodes can only communicate with their neighbors and then necessarily $w_{ij}=0$ if $(i,j) \not \in E$.
Problem~\ref{minim} is called in~\cite{Xiao04} the ``symmetric FDLA problem".

The above minimization problem is a convex one and the function $\mu (W)$ is non-smooth convex function. It is convex since when $W$ is a symmetric matrix, we have $\mu (W)=\rho(W-G_1)=||W-G_1||_2$ which is a composition between an affine function and the matrix L-2 norm, and all matrix norms are convex functions. The function $\mu (W)=\rho(W-G_1)$ is non-smooth since the spectral radius of a matrix is not differentiable at points where the eigenvalues coalesce \cite{Fan:1995}. The process of minimization itself in \eqref{minim} tends to make them coalesce at the solution. Therefore, smooth optimization methods cannot be applied to \eqref{minim}. Moreover, the weight matrix solution of the optimization problem is not unique. For example it can be checked that for the network in Fig.~\ref{conject}, there are infinite weight values that can be assigned to the link $(2,3)$ and solve the optimization problem \eqref{minim}, including $w_{23}=0$. Additionally, this shows that adding an extra link in a graph (e.g.~link $(2,3)$ in the Fig.~\ref{conject}), does not necessarily reduce the second largest eigenvalue of the optimal weight matrix.

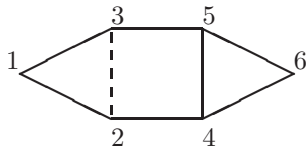
\begin{figure}[h]
\begin{center}

\setlength{\unitlength}{0.6cm} 
\begin{picture}(6,5)
\thicklines 
\put(1,2){\line(2,1){2}}
\put(1,2){\line(2,-1){2}}
\put(0.7,2.1){$1$}

\multiput(3,1.1)(0,0.4){5}
{\line(0,1){0.2}}

\put(3,1){\line(1,0){2}}
\put(3,0.4){$2$}

\put(3,3){\line(1,0){2}} 
\put(3,3.1){$3$}

\put(5,1){\line(0,1){2}}
\put(5,0.4){$4$}

\put(5,3.1){$5$}

\put(7,2){\line(-2,1){2}}
\put(7,2){\line(-2,-1){2}}
\put(7,2.1){$6$}
  \end{picture}
  
  \caption{Network of 6 nodes.}
\label{conject}
\end{center}
\end{figure}

We address in this report a novel approach for the weight selection problem in the average consensus protocol by allowing nodes to optimize a global objective in a totally distributed way. The problem \eqref{minim} is in practice difficult to implement in a distributed way because of the non-smoothness of the function $\mu$. We present in this report a differentiable approximation of the problem, and we show how the new optimization problem can be implemented in a fully decentralized manner using gradient techniques.  We then compare the approximated solution with the optimal one and other distributed weight selection algorithms such as the metropolis or the max degree ones. 

\section{Related Work} \label{relatedwork}
Xiao and Boyd in \cite{Xiao04} have shown that the symmetric FDLA problem \eqref{minim} can be formulated as a Semi-Definite Program (SDP) that can be solved by a centralized unit using interior point methods.
The limit of such centralized approach to weight selection is shown by the fact that a popular solver as \texttt{CVX}, matlab software for disciplined convex programming~\cite{cvx:2011}, can only find the solution of~\eqref{minim} for networks with at most tens of thousands of links. The optimal solution in larger networks can be found iteratively using a centralized subgradient method. A possible approach to distribute the subgradient method and let each node compute its own weights is also proposed in~\cite{Xiao04}, but it requires at each time slot an iterative sub-procedure to calculate an approximation of some eigenvalues and eigenvectors of the matrix $W$ (global information not local to nodes in a network).   

Kim \emph{et al.~} in \cite{Kim:2009} approximate the general FDLA  using the $q$th-order spectral norm (2-norm) minimization ($q$-SNM). They showed that if a symmetric weight matrix is considered, then the solution of the  $q$-SNM is equivalent to that of the symmetric FDLA problem. Their algorithm's complexity is even more expensive than the SDP. Therefore, solving the problem \eqref{minim} in a distributed way is still an open problem.

Some heuristics for the weight selection problem that guarantee convergence of the average protocol and attract some interest in the literature either due to their distributed nature or to the easy implementation are the following (see \cite{Xiao05distributedaverage, Xiao04}): 
\begin{itemize}
\item max degree weights (MD): \\$w_l=\frac{1}{\Delta +1} \ \ \forall l=1...m$.
\item local degree (metropolis) weights (LD): \\$w_l = \frac{1}{\text{max}\{d_i,d_j\}+1} \ l\sim (i,j) \ \ \forall l=1,2,\dots m$. 
\item optimal constant weights (OC): \\$w_l=\frac{2}{\lambda _1(L)+\lambda _{n-1}(L)} \ \ \forall l=1...m.$
\end{itemize}
where $\Delta=\max _i\{d_i\}$ is the maximum degree in the network and $L$ is the Laplacian of the graph. The weight matrix can be then deduced from $\mathbf{w}$: \[W=I-\mathcal{I}\times \text{diag}(\mathbf{w})\times \mathcal{I}^T.\]

 \section{Schatten Norm Minimization}  \label{TM}

We change the original minimization problem in~\eqref{minim} by considering a different cost function that is a monotonic function of the Schatten Norm. 
The minimization problem we propose is the following one:

 \begin{equation}\label{tracemin}
\begin{aligned}
& \underset{W}{\text{Argmin}}
& & f(W)=||W||_{\sigma p}^p \\
& \text{subject to}
& & W=W^T, \\
&&& W\mathbf{1}=\mathbf{1},\\
&&& W\in \mathcal{C} _G,
\end{aligned}
\end{equation}
where $p$ is an even positive integer. The following result establishes that~\eqref{tracemin} is a smooth convex optimization problem and also it provides an alternative expression of the cost function in terms of the trace of $W^p$. For this reason we refer to our problem also as \emph{Trace Minimization} (TM).

\begin{prop}\label{prop1}
$f(W)=||W||_{\sigma p}^p=\Tr(W^p)$ is a scalar-valued smooth convex function on its feasible domain when $p$ is an even positive integer.
\end{prop}
\begin{proof}
We have $\Tr(W^p)=\sum _{i=1}^n\lambda _i^p$. Since $W$ is symmetric, its non-zero singular values are the absolute values of its non-zero eigenvalues~(\cite{Meyer:2000}). Given that  $p$ is even, then $\sum_{i=1}^{n}\lambda _i^p=\sum_{i=1}^n \sigma _i^p$. Therefore, $\Tr(W^p)=||W||_{\sigma p}^p$. 

The Schatten norm $||W||_{\sigma p}$ is a nonnegative convex function, then $f$ is convex because it is the composition of a non-decreasing convex function---function $x^p$ where $x$ is non-negative---and a convex function (see \cite{BoydBook:2004}).  

The function is also differentiable and we have 
\begin{equation}\label{derivativef}
\frac{\text{d}\hfill}{\text{d}W}\Tr(W^p)=pW^{p-1},
\end{equation}
(see \cite[p.~411]{Bernstein:2009}).
\end{proof}

We now illustrate the relation between \eqref{tracemin} and the optimization \eqref{minim}. The following lemmas will prepare the result:

\begin{lemma}
\label{l:asympt}
For any symmetric weight matrix $W$ whose rows (and columns) sum to $1$ and with eigevalues $\lambda_1(W)\ge \lambda_2(W)\ge \dots \ge \lambda_n(W)$, there exist two integers $K_1\in \{1,2,\dots n-1\},K_2 \in \{0,1,2,\dots n-1\}$ and a positive constant $\alpha<1$ such that for any positive integers $p$ and $q$ where $p=2q$ we have:
\begin{equation}
\label{e:asympt}
1+ \tau(W)^p K_1\leq \Tr(W^p) \leq 1+\tau(W)^p (K_1+K_2\alpha^p),
\end{equation} 
where
\begin{equation}
\tau(W) =
\begin{cases}
\rho(W)=\max \{\lambda_1(W),-\lambda_n(W)\} & \text{if } \rho(W) > 1,
\\
\mu(W)=\max\{\lambda_2(W),-\lambda_n(W)\} & \text{if } \rho(W) \leq 1.
\end{cases}
\end{equation}
\end{lemma}
\begin{proof}
Let us consider the matrix $W^2$ and denote by $\nu_1,\nu_2,\dots \nu_r$ its distinct eigenvalues ordered by the largest to the smallest and by $m_1,m_2, \dots m_r$ their respective multiplicities.  We observe that they are all non-negative and then they are also different in module. For convenience we consider $\nu_{s}=m_{s}=0$ for $s>r$.
We can then write:
$$\Tr(W^{p})=\sum_{i=1}^n \lambda_i^p=\sum_{i=1}^r m_i \nu_i^q.$$
The matrix $W^2$ has $1$ as an eigenvalue. Let us denote by $j$ its position in the ordered sequence of distinct eigenvalues, i.e.~$\nu_j=1$. Then it holds:
$$\Tr(W^p)=1+(m_j-1)+ \sum_{i \neq j} m_i \nu_i^q. $$ 

If $\rho(W)=1$ (i.e.~$1$ is the largest eigenvalue in module of $W$), then $1$ is also the largest eigenvalue of $W^2$ ($\nu_1=1$). If $m_1>1$, then it has to be either $\lambda_2(W)=1$ (the multiplicity of the eigenvalue $1$ for $W$ is larger than $1$) or $\lambda_n(W)=-1$. In both cases $\tau(W)=\mu(W)=1$, 
$$\Tr(W^p)=1+(m_1-1)+ \sum_{i > 1} m_i \nu_i^q $$ 
and the result holds with $K_1=m_1-1$, $K_2=\sum_{i>1} m_i$ and $\alpha=\sqrt{\nu_2}$. If $m_1=1$, then $\nu_2=\lambda_2^2$. We can write:
$$\Tr(W^p)=1+\nu_2^q \left(m_2+\sum_{i>2} m_i \left(\frac{\nu_i}{\nu_2}\right)^q \right)$$
and the result holds with $K_1=m_2$, $K_2=\sum_{i>2} m_i$, and $\alpha=\sqrt{\nu_3/\nu_2}$.

If $\rho(W) > 1$, then $\nu_1=\rho(W)^2>1$ and we can write:
$$\Tr(W^p)=1+\nu_1^q \left(m_1+\sum_{\underset{i\neq j}{i>1}} m_i (\frac{\nu_i}{\nu_1})^q + (m_j-1) (\frac{1}{\nu_1})^q \right). $$
Then the result holds with $\tau(W)=\sqrt{\nu_1}=\rho(W)$, $K_1=m_1$, $K_2=\sum_{i>1} m_i$, and $\alpha= \sqrt{\nu_2/\nu_1}$. 
\end{proof}

\begin{lemma}
\label{l:tau}
Let us denote by $W_{(p)}$ the solution of the minimization problem \eqref{tracemin}. If the graph of the network is strongly connected then $\tau\left(W_{(p)}\right) < 1$ for $p$ sufficiently large.
\end{lemma}
\begin{proof}
If the graph is strongly connected then there are multiple ways to assign the weights such that the convergence conditions~\eqref{e:stoch}-\eqref{e:spectral} are satisfied. 
In particular the local degree method described in~Sec.~\ref{relatedwork} is one of them. Let us denote by $W_{(LD)}$ its weight matrix. A consequence of the convergence conditions is that $1$ is a simple eigenvalue of $W_{(LD)}$, and that all other eigenvalues are strictly less than one
in magnitude (see~\cite{Xiao04}). It follows that $\tau\left(W_{(LD)}\right)$ in Lemma~\ref{l:asympt} is strictly smaller than one and that $\lim_{p \to \infty} \Tr\left(W_{(LD)}^p\right)=1$. Then there exists a value $p_0$ such that for each $p>p_0$
$$\Tr\left(W_{(LD)}^p\right)<2.$$
Let us consider the minimization problem~\eqref{tracemin} for a value $p>p_0$. $W_{(LD)}$ is a feasible solution for the problem, then
$$\Tr(W_{(p)}^p)\le \Tr(W_{(LD)}^p)<2.$$
Using this inequality and Lemma \ref{l:asympt}, we have:
$$1+\tau\left(W_{(p)}\right)^p \le 1+\tau\left(W_{(p)}\right)^p K_1\leq \Tr(W_{(p)}^p) < 2,$$
from which the thesis follows immediately.
\end{proof}

We are now ready to state our main results by the following two propositions:
\begin{prop}\label{problem_equivalence}
If the graph of the network is strongly connected, then the solution of the Schatten Norm minimization problem~\eqref{tracemin} satisfies the consensus protocol convergence conditions for $p$ sufficiently large. Moreover as $p$ approaches $\infty$, this minimization problem is equivalent to the minimization problem~\eqref{minim} (i.e.~to minimize the second largest eigenvalue $\mu(W)$).
\end{prop}
\begin{proof}
The solution of problem~\eqref{tracemin}, $W_{(p)}$ is necessarily symmetric and its rows sum to $1$.
From Lemma~\ref{l:tau} it follows that for $p$ sufficiently large $\tau\left(W_{(p)}\right)<1$ then by the definition of $\tau(.)$ it has to be $\rho(W_{(p)})=1$ and $\mu(W_{(p)})<1$. Therefore $W_{(p)}$ satisfies all the three convergence conditions~\eqref{e:stoch}-\eqref{e:spectral} and then the consensus protocol converges.

Now we observe that with respect to the variable weight matrix $W$, minimizing $\Tr(W^p)$ is equivalent to minimizing $(\Tr(W^p)-1)^{1/p}$. From Eq.~\eqref{e:asympt}, it follows:
$$\tau(W)K_1^{\frac{1}{p}}\le (\Tr(W^p)-1)^{\frac{1}{p}}\le \tau(W)(K_1+K_2\alpha^p)^\frac{1}{p}.$$
$K_1$ is bounded between $1$ and $n-1$ and $K_2$ is bounded between $0$ and $n-1$, and $\alpha<1$,then it holds:
$$ \tau(W)K_1^{\frac{1}{p}}\le (\Tr(W^p)-1)^{\frac{1}{p}}\le \tau(W) K^\frac{1}{p},$$
with $K=2(n-1)$.
For $p$ large enough $\tau\left(W_{(p)}\right)=\mu(W_{(p)})$, then 
$$ \left|(\Tr(W_{(p)}^p)-1)^{\frac{1}{p}} - \mu(W_{(p)}) \right| \le \mu(W_{(p)}) \left( K^\frac{1}{p} -1\right)\le  K^\frac{1}{p} -1.$$
Then the difference of the two cost functions converges to zero as $p$ approaches infinity.
\end{proof}

\begin{prop}
The Schatten Norm minimization~\eqref{tracemin} is an approximation for the original problem~\eqref{minim} with a guaranteed error bound, 
$$|\mu (W_{(SDP)})-\mu (W_{(p)})|\leq \mu (W_{(SDP)})\times \epsilon (p),$$
where $\epsilon (p)=(n-1)^{1/p}-1$ and where $W_{(SDP)}$ and $W_{(p)}$ are the solutions of \eqref{minim} and \eqref{tracemin} respectively. 
\end{prop}
\begin{proof}
Let $S$ be the feasibility set of the problem \eqref{minim} (and \eqref{tracemin}), we have $\mu (W)=\max \{\lambda _2(W), -\lambda _n(W)\}$ and let $g(W)=\left(\sum _{i\geq 2}\lambda _i^p(W)\right)^{\frac{1}{p}}$ . Since $W_{(SDP)}$ is a solution of \eqref{minim}, then 
\begin{equation}\label{eq1}
\mu (W_{(SDP)})\leq \mu (W), \ \ \forall W\in S.
\end{equation}
 Note that the minimization of $g(W)$ is equivalent to the minimization of $Tr(W^p)$ when $W\in S$ (i.e. $\underset{W\in S}{\text{Argmin}} \ g(W)=\underset{W\in S}{\text{Argmin}} \ Tr(W^p)$), then 
 \begin{equation}\label{eq2}
 g(W_{(p)})\leq g(W), \ \ \forall W\in S.
 \end{equation}
  Finally for a vector $\mathbf{v} \in \mathbb{R}^m$ all norms are equivalent and in particular $||\mathbf{v}||_\infty \leq ||\mathbf{v}||_p\leq m^{1/p}||\mathbf{v}||_\infty$ for all $p\geq 1$. By applying this inequality to the vector whose elements are the $n-1$ eigenvalues different from $1$ of the matrix $W$, we can write 
 \begin{equation}\label{eq3}
  \mu (W)\leq g(W)\leq (n-1)^{1/p}\mu (W), \ \ \forall W\in S.
 \end{equation}
   Using these three inequalities we can derive the desired bound: 
\begin{align} 
\mu (W_{(SDP)})\overset{\eqref{eq1}}{\leq} \mu (W_{(p)}) \overset{\eqref{eq3}}{\leq} g(W_{(p)}) \overset{\eqref{eq2}}{\leq}  g(W_{(SDP)}) \nonumber\\
\leq (n-1)^{1/p}\mu (W_{(SDP)}),
\end{align}
where the number above the inequalities shows the equation used in deriving the bound.
Therefore $\mu (W_{(SDP)})\leq \mu (W_{(p)}) \leq (n-1)^{1/p}\mu (W_{(SDP)})$ and the proposition directly follows. 
\end{proof}

{\bf Remark:} Comparing the results of Schatten Norm minimization \eqref{tracemin} with the original problem \eqref{minim}, we observe that on some graphs the solution of problem~\eqref{tracemin} already for $p=2$ gives the optimal solution of the main problem \eqref{minim}; this is for example the case  for complete graphs\footnote{
	This can be easily checked. 
	In fact, for any the matrix that guarantees the convergence of average consensus protocols it holds $\mu(W) \ge 0$ and $\Tr(W^2)\ge 1$ (because $1$ is an eigenvalue of $W$). 
	The matrix  $\hat W=1/n \mathbf{1} \mathbf{1}^T$ (corresponding to each link having the same weight $1/n$) has eigenvalues $1$ and $0$ with multiplicity $1$ and $n-1$ respectively. Then  $\mu(\hat W)=0$ and $\Tr(\hat W^2)=1$. It follows that $\hat W$ minimizes both the cost function of problem~\eqref{minim} and~\eqref{tracemin}.
}. However,  on some other graphs, it may give a weight matrix that does not guarantee the convergence of the consensus protocol because the second largest eigenvalue is larger than or equal to $1$ (the other convergence conditions are intrinsically satisfied). We have built a toy example, shown in Fig.~\ref{ToyN}, where this happens. The solution of \eqref{tracemin} assigns weight $0$ to the link $(i,j)$; $w_{ij}=0$ separates the network into two disconnected subgraphs, so $\mu (W)=1$ in this case. We know by Lemma~\ref{l:tau} that this problem cannot occur  for $p$ large enough. In particular for the toy example the matrix solution for $p=4$ already guarantees convergence. We discuss how to guarantee convergence for any value of $p$ in Section \ref{EXT}.

\begin{figure}
\begin{center}
\includegraphics[scale=0.35]{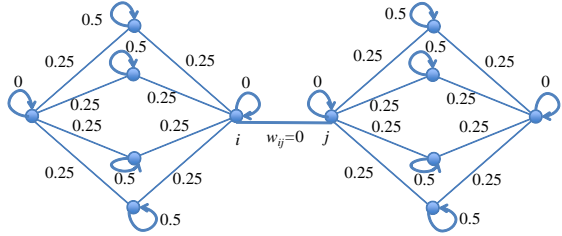}
\caption{For this network the matrix solution of Schatten Norm minimization~\eqref{tracemin} with $p=2$ does not guarantee convergence of average consensus to the true average because $w_{ij}=0$ which separates the network into two parts, each of which can converge to a totally different value (but not to the average of initial values).}
\label{ToyN}
\end{center}
\end{figure}

Given that problem~\eqref{tracemin} is smooth and convex, it can be solved by interior point methods which would be a centralized solution. In the next section we are going to show a distributed algorithm to solve problem~\eqref{tracemin}.

\section{A Distributed Algorithm for Schatten Norm minimization}\label{COTM}


In this section we will show that the optimization problem~\eqref{tracemin} can be solved in a distributed way using gradient methods. 
By distributed algorithm we mean an algorithm where each node only needs to retrieve information from a limited neighborhood (possibly larger than $N_i$) in order to calculate the weights on its incident links.

The constraint $W=W^T$ in the optimization requires any two neighbors $i$ and $j$ to choose the same weight on their common link $l\sim (i,j)$ i.e. $w_{ij}=w_{ji}=w_l$. The last condition  $W\mathbf{1}=\mathbf{1}$ means that at every node $i$ the sum of all weights on its incident links plus its self-weight $w_{ii}$ must be equal to one. This condition is satisfied if nodes choose first weights on links, and then adapt consequently their self-weights $w_{ii}$. Moreover these two constraints lead to the possibility to write $W$ as follows:  $W=I-\mathcal{I}\times \text{diag}(\mathbf{w})\times \mathcal{I}^T$, where $\mathbf{w}\in \mathbb{R}^m$ is the vector of all the weight links $w_l$, $l=1...m$. It follows that Schatten Norm minimization~\eqref{tracemin} is equivalent to the following unconstrained problem:
\begin{equation}\label{unconsTM}
\text{minimize } h(\mathbf{w})=Tr \left((I-\mathcal{I}\times \text {diag}(\mathbf{w})\times \mathcal{I}^T)^p\right).
\end{equation}
We will give a distributed algorithm to solve the Schatten Norm minimization~\eqref{tracemin} by applying gradient techniques to problem \eqref{unconsTM}. Since the cost function to optimize is smooth and convex as we proved in \mbox{Proposition \ref{prop1}}, if the gradient technique converges to a stationary point, then it converges to the global optimum. The gradient method uses the simple iteration:$$w_l^{(k+1)}=w_l^{(k)}-\gamma ^{(k)}g_l^{(k)} \ \forall l=1...m \ ,$$ where $\gamma ^{(k)}$ is the stepsize at iteration $k$ and $g_l^{(k)}$ is the \mbox{$l$-th} component of the gradient $\mathbf{g}^{(k)}$ of the function $h(\mathbf{w})$. At every iteration $k$, starting with a feasible solution for link weights, $w_l^{(k)}$, we calculate the gradient $g_l^{(k)}$ for every link, and then we obtain a new weight value $w_l^{(k+1)}$.

 There are different conditions on the function $h(.)$ and on the stepsize sequence that can guarantee convergence. A distributed computational model for optimizing a sum of non-smooth convex functions is proposed in \cite{Nedic:2009,Lobel11} and its convergence is proved for bounded (sub)gradients for different network dynamics. For a similar objective function, the authors in \cite{Johanson2008} study the convergence of a projected (sub)-gradient method with constant stepsize. For unbounded gradients, the algorithm in \cite[Section~5.3.2, p.~140]{Polyak87} guarantees  global convergence but requires a centralized calculation of the stepsize sequence. Because the objective function in \eqref{unconsTM} has unbounded gradient, our distributed implementation  combines ideas from unbounded gradients methods and the projecting methods using theorems from \cite{bertsekas03}. 
 In particular, we will add a further constraint to \eqref{unconsTM}, looking for a solution in a compact set $X$, and we will consider the following projected gradient method:
$$\mathbf{w}^{(k+1)}= P_X\left(\mathbf{w}^{(k)} - \gamma^{(k)} \mathbf{g}^{(k)}\right),$$
where $P_X()$ is the projection on the set $X$. We can show that by a particular choice of $X$ and $\gamma ^{(k)}$ the method converges to the solution of the original problem. Moreover, all the calculations can be performed in a distributed way on the basis of local knowledge. In particular, we will show that:
\begin{itemize}
\item nodes incident to $l$ are able to calculate $g_l^{(k)}$ using only information they can retrieve from their (possibly extended) neighborhood;
\item the stepsize sequence $\gamma ^{(k)}$ is determined a priori and then nodes do not need  to evaluate the function $h$ or any other global quantity to calculate it; 
\item the projection on set $X$ can be performed component-wise, and locally at each node;
\item the global convergence of the projected gradient method is guaranteed.
\end{itemize} 
  We will start by  $g_l$ and show that it only depends on  information local to nodes $i$ and $j$ incident to the link  $l\sim (i,j)$, then we will discuss the choice of the stepsize $\gamma ^{(k)}$ and of the projection set $X$.

\subsection{Locally Computed Gradient}
Consider the link $l\sim (i,j)$, since $w_l=w_{ij}=w_{ji}$ and $w_{ii}=1-\sum_{s\in N_i}w_{is}$, we have:
\begin{equation}\label{diff}
\frac{\text{d}w_{st}}{\text{d}w_l}=
\begin{cases}
+1 &\text{if } s=i \text{ and } t=j\\
+1 &\text{if } s=j \text{ and } t=i\\
-1 &\text{if } s=i \text{ and } t=i\\
-1 &\text{if } s=j \text{ and } t=j\\
0 &\text{else.}
\end{cases}
\end{equation}
 The gradient $g_l$ of the function $h(\mathbf{w})$ for $l\sim (i,j)$ can be calculated as follows:
\begin{align}
g_l&=\frac{\text{d}\hfill h(\mathbf{w})}{\text{d}w_l}\nonumber \\
&=\frac{\text{d}\hfill f(W)|_{W=I-\mathcal{I}\times \text {diag}(\mathbf{w})\times \mathcal{I}^T}}{\text{d}w_l}\nonumber \\
&=\sum_{s,t}\frac{\partial f}{\partial w_{st}}\frac{\text{d}w_{st}}{\text{d}w_l}\nonumber \\
&=\frac{\partial f}{\partial w_{ij}}\frac{\text{d}w_{ij}}{\text{d}w_l}+\frac{\partial f}{\partial w_{ji}}\frac{\text{d}w_{ji}}{\text{d}w_l}+\frac{\partial f}{\partial w_{ii}}\frac{\text{d}w_{ii}}{\text{d}w_l}+\frac{\partial f}{\partial w_{jj}}\frac{\text{d}w_{jj}}{\text{d}w_l}\nonumber\\
&=\frac{\partial f}{\partial w_{ij}}+\frac{\partial f}{\partial w_{ji}}-\frac{\partial f}{\partial w_{ii}}-\frac{\partial f}{\partial w_{jj}}\nonumber\\
&=p\big ((W^{p-1})_{ji}+(W^{p-1})_{ij}-(W^{p-1})_{ii}-(W^{p-1})_{jj}\big ).\label{line4}
\end{align}
In the last equality we used equation \eqref{derivativef}. 

It is well know from graph theory that if we consider $W$ to be the adjacency matrix of a weighted graph $G$, then $(W^s)_{ij}$ is a function of the weights on the edges of the $i-j$ walks (i.e.~the walks from $i$ to $j$) of length exactly $s$ (in particular if the graph is unweighted $(W^s)_{ij}$ is the number of distinct $i-j$ $s$-walks \cite{West2000}).  Since for a given $p$ the gradient $g_l$, $l\sim (i,j)$, depends on the $\{ii, jj, ij, ji\}$ terms  of the matrix $W^{p-1}$,  $g_l$ can be calculated locally by using only the weights of links and nodes at most $\frac{p}{2}$ hops away from $i$ or $j$\footnote{
	If a link or a node is more than $p/2$ hops away both from node $i$ and node $j$, then it cannot belong to a $i-j$ walk of length $p$.
}
. Practically speaking, at each step, nodes $i$ and $j$ need to contact all the nodes up to $p/2$ hops away in order to retrieve the current values of the weights on the links of these nodes and the values of weights on the nodes themselves. For example, when $p=2$, then the minimization is the same as the minimization of the Frobenius norm of $W$ since $Tr(W^2)=\sum _{i,j}w_{ij}^2=||W||_F^2$, and the  gradient $g_l$ can be calculated as $g_l=2\times (2W_{ij}-W_{ii}-W_{jj})$ which depends only on the weights of the vertices incident to that link and the weight of the link itself. 

An advantage of our approach is that it provides a trade-off between locality and optimality. In fact, the larger the parameter $p$, the better the solution of problem \eqref{tracemin} approximates the solution of problem \eqref{minim}, but at the same time the larger is the neighborhood from which each node needs to retrieve the information. When $p=2$, then $g_l$ where $l\sim (i,j)$ only depends on the weights of subgraph induced by the two nodes $i$ and $j$. For $p=4$, the  gradient $g_l$ depends only on the weights found on the subgraph induced by the set of vertices $N_i\cup N_j$, then it is sufficient that nodes $i$ and $j$ exchange the weights of all their incident links.

\subsection{Choice of Stepsize and Projection set}
The global convergence of gradient methods (i.e.~for any initial condition) has been proved under a variety of different hypotheses on the function $h$ to minimize and on the step size sequence $\gamma^{(k)}$. 
In many cases the step size has to be adaptively selected on the basis of the value of the function or of the module of its gradient at the current estimate, but this cannot be done in a distributed way for the function $h(\mathbf{w})$. This leads us to look for convergence results where the step size sequence can be fixed ahead of time. Moreover the usual conditions, like Lipschitzianity or boundness of the gradient, are not satisfied by the function $h(.)$ over all the feasible set.
For this reason we add another constraint to our original problem~\eqref{unconsTM} by considering that the solution has to belong to a given convex and compact set $X$. Before further specifying how we choose the set $X$, we state our convergence result.

\begin{prop}\label{prop3}
Given the following problem
\begin{eqnarray}\label{unconsTM2}
\text{minimize} & &h(\mathbf{w})=Tr \left((I-\mathcal{I}\times \text {diag}(\mathbf{w})\times \mathcal{I}^T)^p\right),\nonumber\\
\text{subject to} & &\mathbf{w} \in X 
\end{eqnarray}
where $X  \subseteq \mathbb{R}^m $ is a convex and compact set, if $\sum_k \gamma^{(k)}=\infty$ and $\sum_k \left(\gamma^{(k)}\right)^2< \infty$, then the following iterative procedure converges to the minimum of $h$ in $X$:
\begin{equation}\label{e:grad_proj_vect}
\mathbf{w}^{(k+1)}= P_X\left(\mathbf{w}^{(k)} - \gamma^{(k)} \mathbf{g}^{(k)}\right),
\end{equation}
where $P_X(.)$ is the projection operator on the set $X$ and $\mathbf{g}^{(k)}$ is the gradient of $h$ evaluated in $\mathbf{w}^{(k)}$.
\end{prop}
\begin{proof}
The function $h$ is continuous on a compact set $X$, so it has a point of minimum. Moreover also the gradient $\mathbf{g}$ is continuous and then bounded on $X$. The result then follows from Proposition~$8.2.6$ in \cite[pp. 480]{bertsekas03}. 
\end{proof}
For example, $\gamma ^{(k)}=a/(b+k)$ where $a>0$ and $b\geq 0$ satisfies the step size condition in Proposition~\ref{prop3}.

While the convergence is guaranteed for any set $X$ convex and compact, we have two other requirements. First, it should be possible to calculate the projection $P_X$ in a distributed way. Second, the set $X$ should contain the solution of the optimization problem~\eqref{unconsTM}.
About the first issue, we observe that if $X$ is the cartesian product of real intervals, i.e. if $X=[a_1,b_1]\times[a_2,b_2]\times \dots [a_m,b_m]$, then we have that the $l$-th component of the projection on $X$ of a vector $\mathbf{y}$ is simply the projection of the $l$-th component of the vector on the interval $[a_l,b_l]$, i.e.:
\begin{equation}
\left[P_X(\mathbf{y})\right]_l=P_{[a_l,b_l]}(y_l)=
\begin{cases}
a_l&\text{if } y_l<a_l,\\
y_l&\text{if } a_l \leq y_l\leq b_l,\\
b_l&\text{if } b_l<y_l.
\end{cases}
\end{equation} 
Then in this case Eq.~\eqref{e:grad_proj_vect} can be written component-wise as
$$w_l^{(k+1)}= P_{[a_l,b_l]}(w_l^{(k)} - \gamma^{(k)} g_l^{(k)}).$$
We have shown in the previous section that $g_l$ can be calculated in a distributed way, then the iterative procedure can be distributed.
About the second issue, we choose $X$ in such a way that we include in the feasibility set all the weight matrices with spectral radius at most $1$. 
The following lemma indicates  how to choose $X$.
\begin{lemma} \label{l:proj}
Let $W$ be a real and symmetric matrix where each row (and column) sums to $1$, then the following holds,$$\rho (W)=1 \ \Longrightarrow \ \max _{i,j}|w_{ij}|\leq 1.$$
\end{lemma}
\begin{proof}
Since $W$ is real and symmetric, then we can write $W$ as follows 
$$W=S\Lambda S^T,$$ 
where $S$ is an orthonormal matrix ($S^TS=SS^T=I$), and $\Lambda$ is a diagonal matrix having $\Lambda _{kk}=\lambda _k$ and $\lambda _k$ is the $k$-th largest eigenvalue of $W$. Let $\mathbf{r}_k$ and $\mathbf{c}_k$ be the rows and columns of $S$ respectively and $r_k^{(i)}$ be the $i$-th element of this vector. So, $$W=\sum _k\lambda _k\mathbf{c}_k\mathbf{c}_k^T,$$ and
\begin{align}
|w_{ij}|&=|\sum _k\lambda _kc_k^{(i)}c_k^{(j)}|\label{Teq1}\\
&\leq \sum _k|c_k^{(i)}||c_k^{(j)}|\label{Teq2}\\
&= \sum _k|r_i^{(k)}||r_j^{(k)}|\label{Teq3}\\
&\leq ||\mathbf{r}_i||_2||\mathbf{r}_j||_2\label{Teq4}\\
&=1.\label{Teq5}
\end{align}
 The transition from \eqref{Teq1} to \eqref{Teq2} is due to the fact $\rho (W)=1$, the transition from \eqref{Teq3} to \eqref{Teq4} is due to Cauchy--Schwarz inequality. The transition from \eqref{Teq4} to \eqref{Teq5} is due to the fact that $S$ is an orthonormal matrix.
\end{proof}

A consequence of Lemma~\ref{l:proj} is that if we choose $X=[-1,1]^m$ the weight vector of the matrix solution of problem~\eqref{minim} necessarily belongs to $X$ (the weight matrix satisfies the convergence conditions). The same is true for the solution of problem~\eqref{unconsTM} for $p$ large enough because of Proposition~\ref{problem_equivalence}. The following proposition summarizes our results.

\begin{prop}
If the graph of the network is strongly connected, then the following distributed algorithm converges to the solution of the Schatten norm minimization problem for $p$ large enough:
\begin{equation}
w_l^{(k+1)}= P_{[-1,1]}(w_l^{(k)} - \gamma^{(k)} g_l^{(k)}), \;\; \forall l =1, \dots, m,
\end{equation}
where $\sum_k \gamma^{(k)}=\infty$ and $\sum_k \left(\gamma^{(k)}\right)^2< \infty$.
\end{prop}

\subsection{Complexity of the Algorithm}

Our distributed algorithm for Schatten Norm minimization requires to calculate at every iteration, the stepsize $\gamma ^{(k)}$, the gradient $g_l^{(k)}$ for every link, and a projection on the feasible set $X$. Its complexity is determined by the calculation of link gradient $g_l$, while the cost of the other operations is negligible. 
 In what follows, we detail the computational costs (in terms of number of operations and memory requirements) and communication costs (in terms of volume of information to transmit) incurred by each node for the optimization with the two values $p=2$ and $p=4$.

\subsubsection{Complexity for $p=2$}

For $p=2$, $g_l=2\times (2W_{ij}-W_{ii}-W_{jj})$, so taking into consideration that nodes are aware of their own weights ($W_{ii}$) and of the weights of the links they are incident to ($W_{ij}$), the only missing parameter in the equation is their neighbors self weight ($W_{jj}$). So at every iteration of the subgradient method, nodes must broadcast their self weight to their neighbors. We can say that the computational complexity for $p=2$ is negligible and the communication complexity is $1$ message carrying a single real value ($w_{ii}$) per link, per node and per iteration.

\subsubsection{Complexity for $p=4$}

For $p=4$, the node must collect information from a larger neighborhood. The gradient at link $l\sim (i,j)$ is given by $g_l=4\big ( (W^3)_{ij}+(W^3)_{ji}-(W^3)_{ii}-(W^3)_{jj}\big )$.
From the equation of $g_l$ it seems like the node must be aware of all the weight matrix in order to calculate the 4 terms in the equation, however this is not true. As discussed in the previous section, each of the 4 terms can be calculated only locally from the weights within 2-hops from $i$ or $j$. In fact, $(W^3)_{ij}$  depends only on the weights of links covered by a walk with 3 jumps: Starting from $i$ the first jump reaches a neighbor of $i$, the second one a neighbor of $j$ and finally the third jump finishes at $j$, then we cannot move farther than 2 hops from $i$. Then this term can be calculated at node $i$ as follows: Every node $s$ in $N_i$, sends its weight vector $\mathbf{W}_s$ to $i$ ($\mathbf{W}_s$ is a vector that contains all weights selected by node $s$ to its neighbors). The same is true for the addend $(W^3)_{ji}$. The term $(W^3)_{ii}$ depends on the walks of length 3 starting and finishing in $i$, then node $i$ can calculate it once it knows $\mathbf{W}_s$ for each $s$ in $N_i$. Finally, the calculation of the term $(W^3)_{jj}$ at node $i$ requires $i$ to know more information about the links existing among the neighbors of node $j$. Instead of the transmission of this detailed information, we observe that node $j$ can calculate the value $(W^3)_{jj}$ (as node $i$ can calculate $(W^3)_{ii}$) and then can transmit directly the result of the calculation to node $i$.  Therefore, the calculation of $g_l$ by node $i$ for every link $l$ incident to $i$ can be done in three steps:
\begin{enumerate}
\item Create the subgraph $H_i$ containing the neighbors of $i$ and the neighbors of its neighbors by sending ($\mathbf{W}_i$) and receiving the weight vectors ($\mathbf{W}_s$) from every neighbor $s$.
\item Calculate $(W^3)_{ii}$ and broadcast it to the neighbors (and receive $(W^3)_{ss}$ from every neighbor $s$).
\item Calculate $g_l$.
\end{enumerate}
We evaluate now both the computational and the communication complexity.

\begin{itemize}
\item Computation Complexity: Each node $i$ must store the subgraph $H_i$ of its neighborhood. The number of nodes of $H_i$ is $n_H\leq \Delta ^2 +1$, the number of links of $H_i$ is $m_H\leq \Delta ^2$ where $\Delta$ is the maximum degree in the network. Due to sparsity of matrix $W$, the calculation of the value $(W^3)_{ii}$ requires $O(\Delta ^3)$ multiplication operation without the use of any accelerating technique in matrix multiplication which ---we believe--- could further reduce the cost. So the total cost for calculating $g_l$ is in the worst case $O(\Delta ^3)$. Notice that the complexity for solving the SDP for \eqref{minim} is of order $O(m^3)$ where $m$ is the number of links in the network. Therefore, on networks where $\Delta << m$, the gradient method would be computationally more efficient. 
\item Communication Complexity: Two packets are transmitted by each node on each link at steps $1$ and $2$. So the complexity would be two messages per link per node and per iteration. The first message carries at most $\Delta$ values (the weight vector $\mathbf{W}_i$) and the second message carries one real value ($(W^3)_{ii}$).
\end{itemize}

\section{Performance Evaluation}\label{PE}
In this section we evaluate the speed of convergence of consensus protocols when the weight matrix $W$ is selected according to our algorithm. As we have discussed in Section~\ref{s:fastest}, this speed is asymptotically determined by the second largest eigenvalue in module ($\mu(W)$), that will be one of two performance metrics considered here. The other will be defined later. The simulations are done on random graphs (Erd\"os-Renyi (ER) graphs and Random Geometric Graphs (RGG)) and on two real networks (the Enron company internal email exchange network \cite{enron} and the dolphin social network \cite{dolphin}). The random graphs are generated as following : 
\begin{itemize}
\item For the ER random graphs, we start from $n$ nodes fully connected graph, and then every link is removed from the graph by a probability $1-Pr$ and is left there with a probability $Pr$. We have tested the performance for different probabilities $Pr$.
\item For the RGG random graphs, $n$ nodes are thrown uniformly at random on a unit square area, and any two nodes within a connectivity radius $r$ are connected by a link. We have tested the performance for different connectivity radii. It is known that for a small connectivity radius, the nodes tend to form clusters. 
\end{itemize}
 The real networks are described as following:
 \begin{itemize}
 \item The Enron company has 151 employees where an edge in the graph refers to an exchange of emails between two employees (only internal emails within the company are considered where at least 3 emails are exchanged between two nodes in this graph).
 \item The dolphin social network is an undirected social network of frequent associations between 62 dolphins in a community living off Doubtful Sound, New Zealand.
 \end{itemize}

\subsection{Comparison with the optimal solution}
We first compare $\mu\left(W_{(p)}\right)$ for the solution $W_{(p)}$ of the Schatten p-norm (or Trace) minimization problem~\eqref{tracemin} with its minimum value obtained solving the symmetric FDLA problem~\eqref{minim}. To this purpose we used the \texttt{CVX} solver (see section~\ref{relatedwork}). This  allows us also to evaluate how well problem~\eqref{tracemin} approximates problem~\eqref{minim} for finite values of the parameter $p$. 
The results in Fig.~\ref{ERfdla} have been averaged over $100$ random graphs with $20$ nodes generated according to the Erdos-Renyi (ER) model, where each link is included with probability $Pr \in \{0.2,0.3,0.4,0.5\}$.   We see from the results that as we solve the trace minimization for larger $p$, the asymptotic convergence speed of our approach converges to the optimal one as proven in Proposition \ref{problem_equivalence}.

 \begin{figure}
\begin{center}
\includegraphics[scale=0.35]{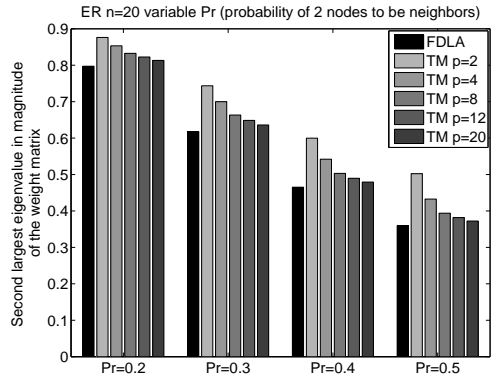}
\caption{Performance comparison between the optimal solution of the FDLA problem (labeled FDLA) and the approximated solutions obtained solving the Schatten Norm minimization for different values of $p$ (labeled TM).}
\label{ERfdla}
\end{center}
\end{figure}


\subsection{Other distributed approaches: Asymptotic Convergence Rate}

We compare now our algorithm for $p=2$ and $p=4$ with other distributed weight selection approaches described in section~\ref{relatedwork}.

Fig.~\ref{ERtrace4} shows the results on connected Erd\"os-Renyi (ER) graphs and Random Geometric Graphs (RGG) with $100$ nodes for different values respectively of the probability $Pr$ and of the connectivity radius $r$. 
We provide 95\% confidence intervals by averaging each metric over $100$ different samples.
\begin{figure}
\begin{center}
\includegraphics[scale=0.35]{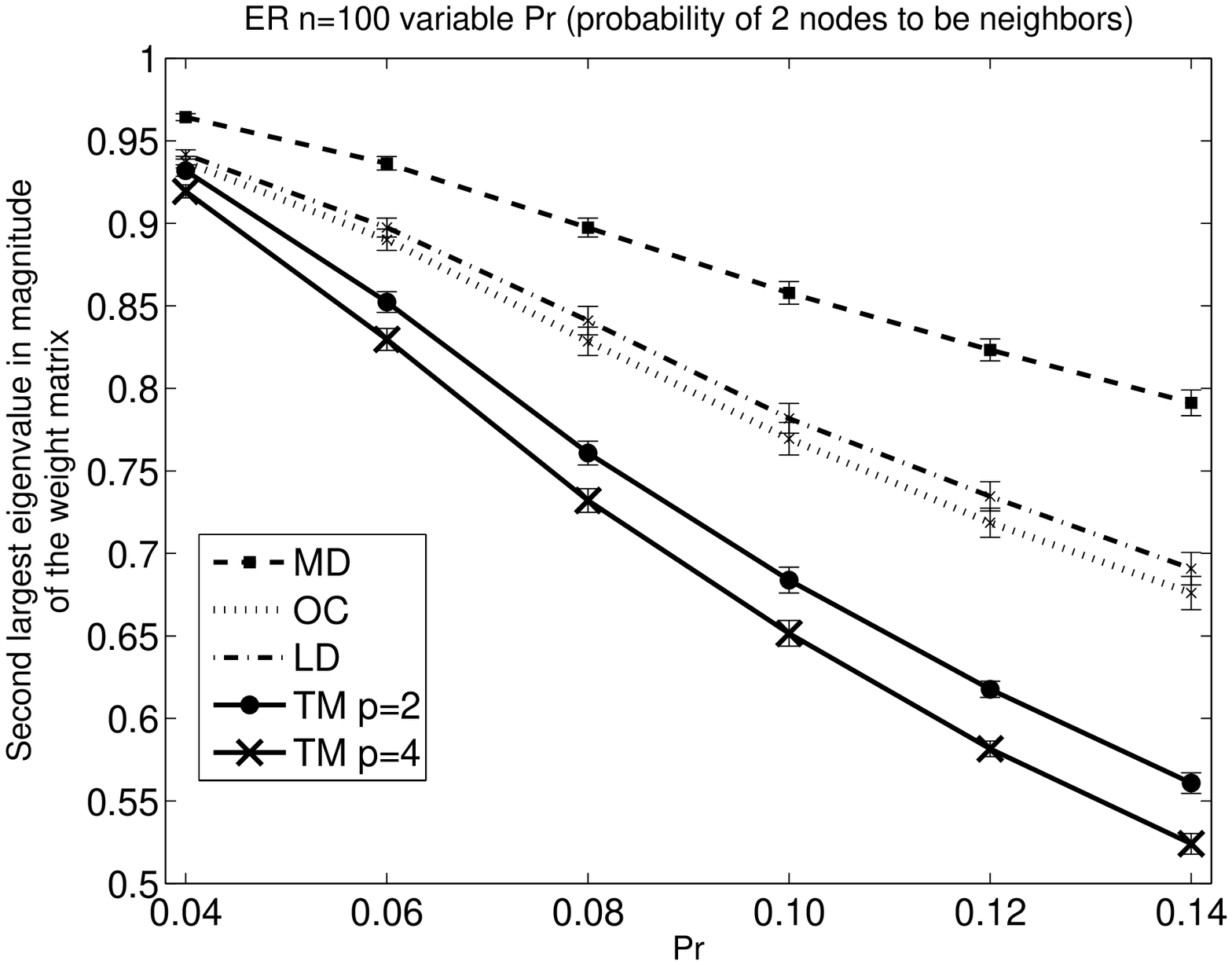}
\includegraphics[scale=0.35]{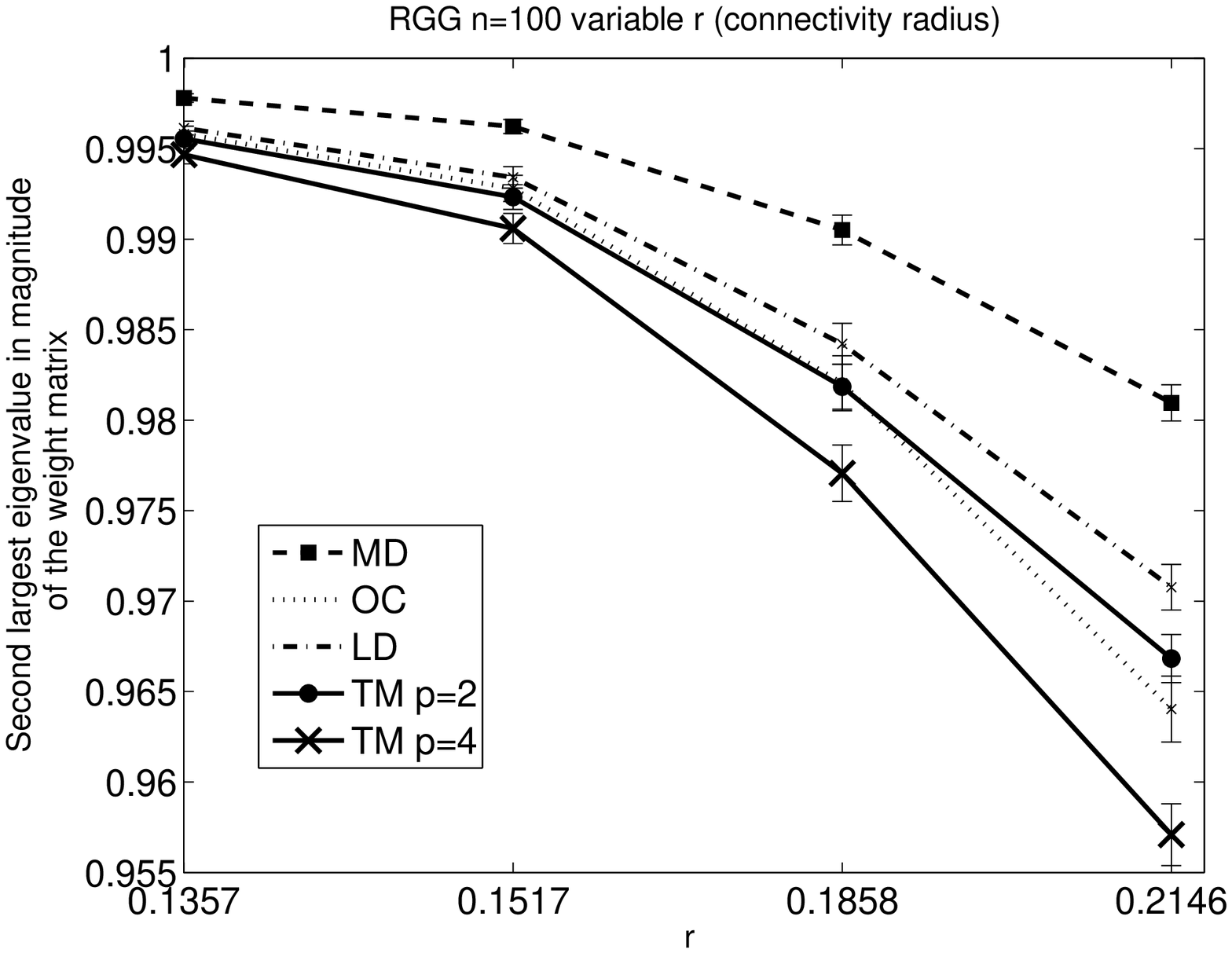}
\caption{Performance comparison between Schatten Norm minimization (TM) for $p=2$ and $p=4$ with other weight selection algorithms on ER and RGG graphs.}
\label{RGGtrace4}
\label{ERtrace4}
\end{center}
\end{figure}
We see in Fig.~\ref{RGGtrace4} that TM for $p=2$ and $p=4$ outperforms other weight selection algorithms on ER by giving lower $\mu$. Similarly  on RGG the TM algorithm reaches faster convergence than the other known algorithms even when the graph is well connected (large connectivity radius). However, the larger the degrees of nodes, the higher the complexity of our algorithm. Interestingly even performing trace minimization for the smallest value $p=2$ nodes are able to achieve faster speed of convergence than a centralized solution like the OC algorithm.

Apart from random networks, we performed simulations on two real world networks: the Enron company internal email exchange network \cite{enron} and the dolphin social network \cite{dolphin}. 
The table below compares the second largest eigenvalue $\mu$ for the different weight selection algorithms on these networks:

\begin{table}[h]
\centering
    \begin{tabular}{|l|c|c|c|c|c|}
        \hline
        ~       & MD & OC & LD & TM p=2 &   TM p=4 \\ \hline
        Enron $\mu$  & 0.9880  & 0.9764  & 0.9862  & 0.9576      & 0.9246        \\  \hline
        Dolphin $\mu$ & 0.9867  & 0.9749  & 0.9796  & 0.9751      & 0.9712        \\
        \hline
    \end{tabular}
\end{table}

The results show that for Enron network, our totally distributed proposed algorithm TM for p=4  has the best performance ($\mu=0.9246$) among the studied weight selection algorithms followed by TM for p=2 ($\mu=0.9576$) because they have the smallest $\mu$. On the Dolphin's network, again TM for p=4 had the smallest $\mu$ ($\mu =0.9712$)  but OC had the second best performance ($\mu =0.9749$) where TM for p=2 ($\mu =0.9751$) was close to the OC performance. 

\subsection{Communication Overhead for Local Algorithms  }

Until now we evaluated only the asymptotic speed of convergence, independent from the initial values $x_i(0)$, by considering the second largest eigenvalue $\mu(W)$.
We want to study now the transient performance.
For this reason, we consider  a random initial distribution of nodes' values and we define the convergence time to be the number of iterations needed for the error (the distance between the estimates and the actual average) to become smaller than a given threshold. More precisely, we define the normalized error $e(k)$ as 
\begin{equation}\label{consensuserror}
e(k)=\frac{{||\mathbf{x}(k)-\bar{\mathbf{x}}||}_2}{{||\mathbf{x}(0)-\bar{\mathbf{x}}||}_2},
\end{equation}
where $\bar{\mathbf{x}}=x_{ave}\mathbf{1}$, and the convergence time is the minimum number of iterations after which $e(k)<0.001$ (note that $e(k)$ is non increasing).

As the Schatten norm minimization problem itself may take a long time to converge, whereas other heuristics can be obtained instantaneously, the complexity of the optimization algorithm can affect the overall procedure. If we consider a fixed network (without changes in the topology), the weight optimization procedure is done before the start of the consensus cycles,\footnote{For example, the cycle of the daily average temperature in a network of wireless environmental monitoring sensors is one day because every day a new averaging consensus algorithm should be run.} and then the same weights are used for further average consensus cycles. Therefore, the more stable the network, the more one is ready to invest for the optimization at the beginning of consensus. The communication overhead of the local algorithms is plotted in Fig.~\ref{Complexity}. For each algorithm we consider the following criteria to define its communication overhead. First we consider the number of messages that should be exchanged in the network for the weight optimization algorithm to converge. For example, in our networking settings (RGG with $100$ nodes and connectivity radius $0.1517$) the initialization complexity of MD algorithm is 30 messages per link because the maximum degree can be obtained by running a maximum consensus algorithm that converges after a number of iterations equal to the diameter (the average diameter for the graphs was 15 hops), while by LD the nodes only need to send their degrees to their neighbors which makes its complexity for establishing weights only 2 messages per link which is the least complexity among other algorithms. The trace minimization algorithm complexity is defined by the number of iterations needed for the gradient method to converge, multiplied by the number of messages needed per iteration as mentioned in the complexity section. In our networking setting, the $TM$ for $p=2$ took on average $66.22$ messages per link to converge while the $TM$ for $p=4$ took $1388.28$ messages.\footnote{The step size $\gamma _k$ is calculated with values $a=10/p$ and $b=100$, and convergence is obtained when $||g||$ drops below the value $0.02$.} Notice that OC depends on global values (eigenvalues of the laplacian of the graph) and is not included here because it is not a local algorithm and cannot be calculated with simple iterative local methods. 

In addition to the initialization complexity, we add the communication complexity for the consensus cycles. We consider that the convergence of the consensus is reached when the consensus error of Eq.~\eqref{consensuserror} drops below $0.1\%$. The results of Fig.~\ref{Complexity} show that if the network is used for 1 or 2 cycles the best algorithm is to use $TM$ for $p=2$, followed by $LD$, followed by $MD$, and the worst  overhead is for $TM$ for $p=4$. If the network is used between 3 and 5 cycles, then $TM$ where $p=4$ becomes better that $MD$ but still worst than the other two. Further more, the $TM$ where $p=4$ becomes better than $LD$ for the 6th and 7th cycles. And finally, if the network is stable for more than 7 cycles, the $TM$ for $p=4$ becomes the best as the asymptotic study shows.  
  
\begin{figure}
\begin{center}
\includegraphics[scale=0.35]{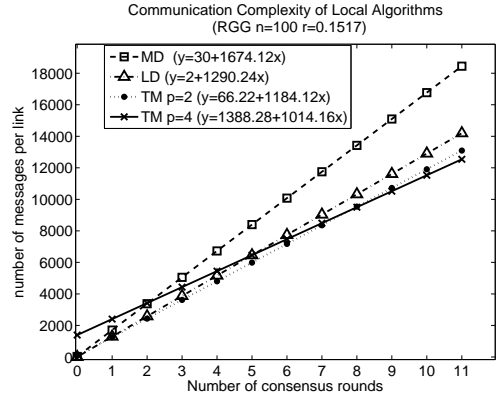}
\caption{Communication overhead of local algorithms.
} 
\label{Complexity}
\end{center}
\end{figure}

\subsection{Joint Consensus-Optimization (JCO) Procedure}\label{JCO}

In the following experiments we address also another practical concern. It may seem our approach requires to wait for the convergence of the iterative weight selection algorithm before being able to run the consensus protocol.
This may be unacceptable in some applications specially if the network is dynamic and the weights need to be calculated multiple times. In reality, at each slot the output of the distributed Schatten norm minimization is a new feasible weight matrix, that can be used by the consensus protocol, and (secondarily)  should also have faster convergence properties than the one at the previous step. It is then possible to interleave the weight optimization steps and the consensus averaging ones: at a given slot each node will improve its own weight according to \eqref{line4} and use the current weight values to perform the averaging \eqref{wsum}. We refer to this algorithm as the joint consensus--optimization (JCO) procedure. Weights can be initially set according to one of the other existing algorithms like LD or MD. The convergence time of JCO depends also on the choice of the stepsize, that is chosen to be $\gamma ^{(k)}=\frac{1}{p(1+k)}$. 

\begin{figure}
\begin{center}
\includegraphics[scale=0.35]{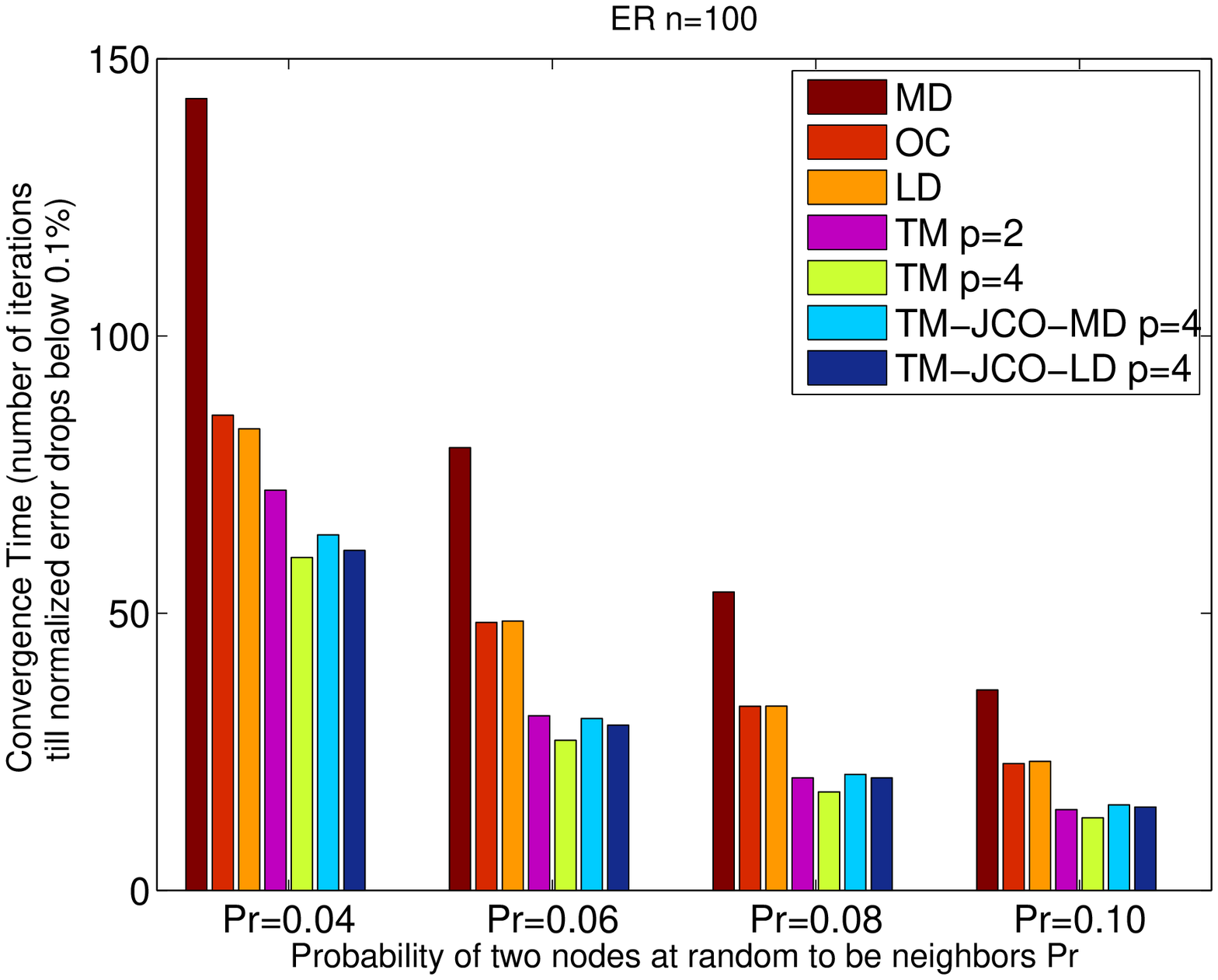}
\includegraphics[scale=0.35]{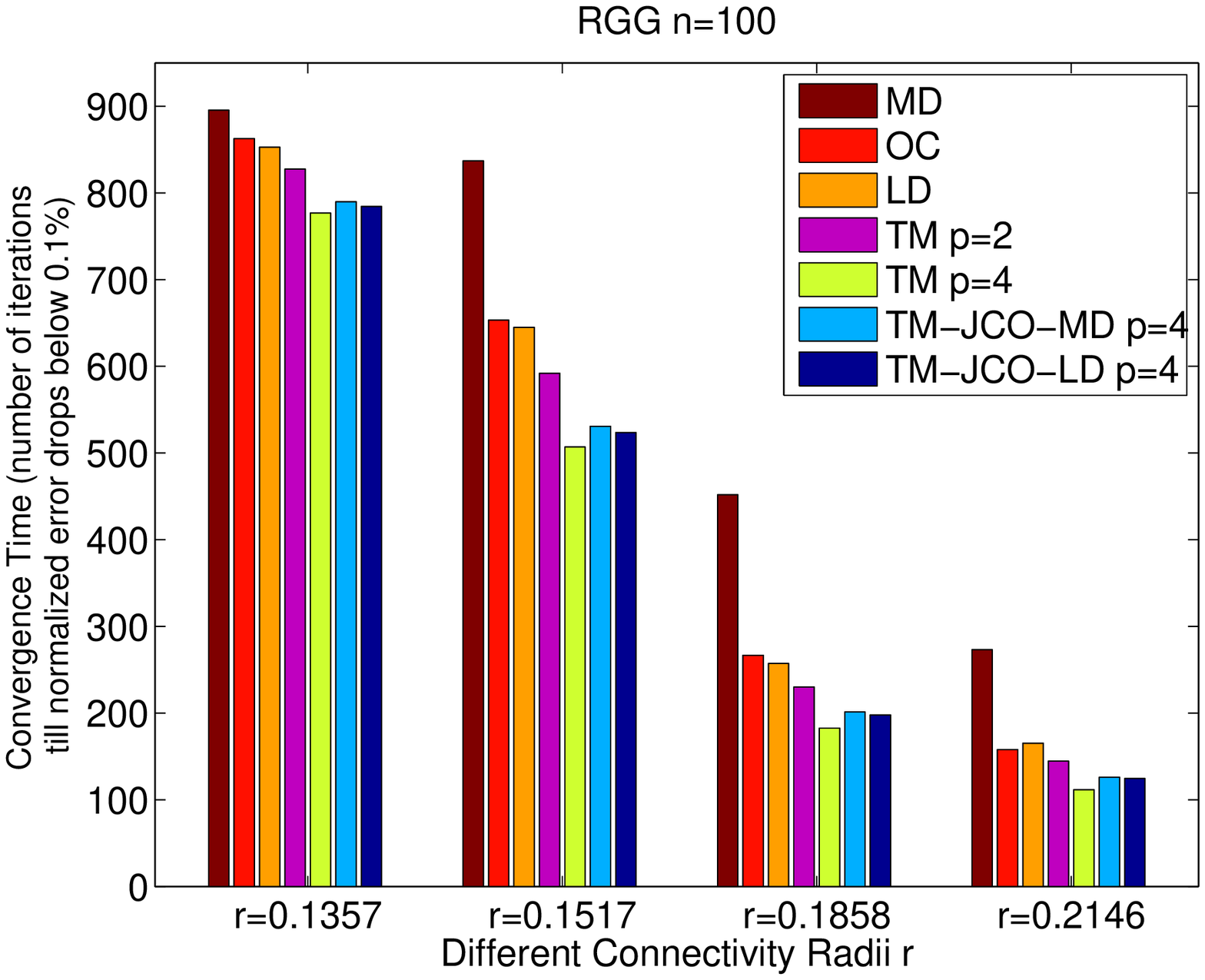}
\caption{Convergence time of different weight selection algorithms on ER and RGG graphs. TM-JCO-LD $p=4$ is the joint consensus-optimization algorithm initialized with the LD algorithm's weight matrix and the same for TM-JCO-MD $p=4$ but initialized with the MD algorithm's one.
} 
\label{RGGtrace4_sim}
\label{ERtrace4_sim}
\end{center}
\end{figure}
The simulations show that our weight selection algorithm outperforms the other algorithms also in this case. In particular, Fig.~\ref{RGGtrace4_sim}  shows the convergence time for various weight selection criteria on ER and RGG graphs. For each of the network topology selected, we averaged the data in the simulation over 100 generated graphs, and for each of these graphs we averaged the convergence time of the different algorithms over 20 random initial conditions (the initial conditions were the same for all algorithms). Notice that running at the same time the optimization with consensus gave good results in comparison to LD, MD, and even OC algorithms. We also notice, that the initial selection of the weights does not seem to have an important role for the TM-JCO approach. In fact, despite the LD weight matrix leads itself a significantly faster convergence than the MD weight matrix, the difference between TM-JCO-MD and TM-JCO-LD is minor, suggesting that the weight optimization algorithm moves fast away from the initial condition.

\section{Stability and Misbehaving Nodes}
\label{EXT}
In this section we first explain how the convergence of the consensus protocol can be guaranteed also for ``small" $p$ values (see the remark in section~\ref{TM}) and then we discuss how to deal with some forms of nodes' misbehavior.

\subsection{Guaranteeing Convergence of Trace Minimization}

The conditions \eqref{e:stoch}-\eqref{e:spectral}  guarantee that the consensus protocol converges to the correct average independently from the initial estimates.
In this section, for the sake of conciseness, we call a weight matrix that satisfies these set of conditions a \emph{convergent matrix}. A convergent matrix is any matrix that guarantees the convergence of average consensus protocols. We showed in Proposition~\ref{problem_equivalence} that for $p$ large enough, the solution $W_{(p)}$ of \eqref{tracemin} is a convergent matrix.  However, for ``small"~$p$ values, it may happen that $\mu(W_{(p)})\ge 1$  (the other conditions are intrinsically  satisfied) and then the consensus protocol does not converge for all the possible initial conditions. We observe that if all the link weights and the self weights in $W_{(p)}$ are strictly positive then  $W_{(p)}$ is a convergent matrix. In fact from Perron-Frobenius theorem for nonnegative matrices \cite{seneta2006non} it follows that a stochastic weight matrix $W$ for a strongly connected graph where $w_{ij}>0$ if and only if $(i,j)\in E$ satisfies \eqref{e:spectral} (i.e. $\mu (W)< 1$). 
Then, the matrix may not be convergent only if one of the weights is negative. 
Still in such a case nodes can calculate in a distributed way a convergent weight matrix that is ``close" to the matrix $W_{(p)}$. In this section we show how it is possible and then we discuss a practical approach to guarantee convergence while not sacrificing  the speed of convergence of $W_{(p)}$ (when it converges).

We obtain a convergent matrix from $W_{(p)}$ in two steps. First, we project  $W_{(p)}$ on a suitable set of  matrices that satisfy conditions \eqref{e:stoch} and \eqref{e:spectral}, but not necessarily \eqref{e:stoch2}, then we generate a symmetric convergent matrix from the projection.
Let $\hat{W}=W_{(p)}$ be the matrix to project, the solution of the following projection is guaranteed to satisfy~\eqref{e:stoch} and~\eqref{e:spectral}:
\begin{equation}\label{projTM}
\begin{aligned}
& \underset{W}{\text{Argmin}}
& & ||W-\hat{W}||_F^2\\
& \text{subject to}
& & W\mathbf{1}=\mathbf{1},\\
&&& W\in \mathcal{C} _G^\prime,
\end{aligned}
\end{equation}
where $\mathcal{C} _G^\prime$ is the set of non-negative matrices such that $w_{ij}\ge \delta> 0$ if $(i,j)\in E$, $w_{ij}=0$ if $(i,j) \notin E$, and $||.||_F$ is the Frobenius matrix norm. 
The constant $\delta>0$ is a parameter that is required to guarantee that the feasible set is closed.

Now, we show how it is possible to project a matrix $\hat W$ according to \eqref{projTM} in a distributed way. 
We observe that this approach is feasible because we do not require the projected matrix to be symmetric (and then satisfy \eqref{e:stoch2}). 
The key element for the distributed projection is that the Frobenius norm is separable in terms of the variables $\mathbf{W}_i$ (the $d_i\times 1$ vector of weights selected by node $i$ for its neighbors), so that problem~\eqref{projTM} is equivalent to:
  \begin{equation}\label{localprojTM}
\begin{aligned}
& \underset{\mathbf{W}_1,...,\mathbf{W}_n}{\text{Argmin}}
& & \sum _{i=1}^nr(\mathbf{W}_i)\\
& \text{subject to}
& & \mathbf{W}_i^T\mathbf{1}_{d_i}\leq 1 \ \ \forall i,\\
&&& \mathbf{W}_i \ge \delta >0 \ \ \forall i,
\end{aligned}
\end{equation}
where $\mathbf{1}_{d_i}$ is the $d_i\times 1$ vector of all ones, 
and $r(\mathbf{W}_i)$ is defined as follows:

\begin{align}
r(\mathbf{W}_i)&=(w_{ii}-\hat{w}_{ii})^2+\sum _{j\in N_i}(w_{ij}-\hat{w}_{ij})^2\\
&=(\mathbf{W}_i-\mathbf{\hat{W}}_i)^T(\mathbf{W}_i-\mathbf{\hat{W}}_i)+\left((\mathbf{W}_i-\mathbf{\hat{W}}_i)^T\mathbf{1}_{d_i}\right)^2\\
&=(\mathbf{W}_i-\mathbf{\hat{W}}_i)^T\left(I_{d_i}+\mathbf{1}_{d_i}\mathbf{1}^T_{d_i}\right)(\mathbf{W}_i-\mathbf{\hat{W}}_i),
\end{align}
where $I_{d_i}$ is $d_i$-identity matrix.
Since the variables in \eqref{localprojTM} are separable in $\mathbf{W}_1, ...,\mathbf{W}_n$, then each node $i$ can find the global solution for its projected vector $\mathbf{W}_i^{(proj)}$ by locally minimizing the function $r(\mathbf{W}_i)$ subject to its constraints. 

Once the weight vectors $\mathbf{W}_i^{(proj)}$ are obtained, the projection of $W_{(p)}$ on the set $\mathcal{C} _G^\prime$ is uniquely identified. We denote it $W^{(proj)}$. We can then obtain a convergent weight matrix $W^{(conv)}$ by modifying $W^{(proj)}$  as follows.
For every link $l\sim (i,j)$, we set:
$$w^{(conv)}_l=\min\left\{\left(\mathbf{W}^{(proj)}_i\right)_{\alpha(j)},\left(\mathbf{W}^{(proj)}_j\right)_{\alpha(i)}\right\},$$ 
where $\alpha(j)$ (similarly $\alpha(i)$) is the index of the node $j$ (similarly $i$) in the corresponding vector. Then we calculate the convergent weight matrix: $$W^{(conv)}=I-\mathcal{I}\times \text{diag}(\mathbf{w^{(conv)}})\times \mathcal{I}^T.$$ 

While the matrix $W^{(conv)}$ is convergent, its speed of convergence may be slower than the matrix $W_{(p)}$, assuming this converges too. Then the algorithm described above should  be ideally limited to the cases where $W_{(p)}$ is known to not be convergent. Unfortunately in many network scenarios this may not be known a priori.
We discuss a possible practical approach in such cases. 
Nodes first compute $W_{(p)}$. If all the link-weights and self-weights are positive then the matrix $W_{(p)}$ can be used in the consensus protocol without any risk. If one node has calculated a non-positive weight, then it can invoke the  procedure described above to calculate $W^{(conv)}$. Nodes can then run the consensus protocol using only the matrix $W^{(conv)}$ at the price of a slower convergence or they can run the two consensus protocols in parallel averaging the initial values both with $W^{(conv)}$ and $W_{(p)}$. If the estimates obtained using $W_{(p)}$ appear to be converging to the same value of the estimates obtained using $W^{(conv)}$, then the matrix $W_{(p)}$ is likely to be convergent and the corresponding estimates should be closer to the actual average\footnote{
	Note that if $\mu(W_{(p)})>1$  the estimates calculated using $W_{(p)}$ diverge in general, then it should be easy to detect that the two consensus protocols are not converging to the same value.
}.


\subsection{Networks with Misbehaving Nodes}

 The convergence of the average consensus relies on all the nodes correctly performing the algorithm.
 If one node transmits an incorrect value, the estimates of all the nodes can be affected.
In this section we address this particular misbehavior. In particular, let $x_i(k)$ be the estimate of node $i$ at iteration $k$, if $x_i(k)\neq w_{ii}(k-1)x_i(k-1)+\sum _{j\in N_i}w_{ij}(k-1)x_j(k-1)$, then we call $i$ a \emph{misbehaving node}. \emph{Stubborn} nodes are a special class of misbehaving nodes that keep sending the same estimate at every iteration (i.e. a node $i$ is a stubborn node when at every iteration $k$ we have $x_i(k)=x_i(k-1)\neq w_{ii}(k-1)x_i(k-1)+\sum _{j\in N_i}w_{ij}(k-1)x_j(k-1)$). The authors of \cite{Acemoglu2011} and \cite{Ben-Ameur2012} showed that networks with stubborn nodes fail to converge to consensus. In \cite{Ben-Ameur2012}, they proposed a robust average consensus algorithm that can be applied on networks having one stubborn node and converges to consensus. To the best of our knowledge, dealing with multiple stubborn nodes is still an open issue. It turns out that with a minor modification of our JCO algorithm, the nodes can detect an unbounded number of misbehaving nodes under the following assumptions:
 \begin{itemize}
 \item {\bf Assumption 1:} There is no collusion between misbehaving nodes (every node that detects a misbehaving neighbor declares it).
 \item {\bf Assumption 2:} At each iteration a misbehaving node sends the same (potentially wrong) estimate to all its neighbors. 
\end{itemize}
The second assumption can be automatically satisfied in the case of a broadcast medium.

In the JCO procedure in section \ref{JCO}, nodes perform one weight optimization step and one average consensus step at every iteration. Consider an iteration $k$, weight optimization requires nodes to receive the weight vectors used by their neighbors (in particular, node $i$ will receive $\mathbf{W}_j^{(k-1)}$ from every neighbor $j\in N_i$), and the averaging protocol requires them to receive their neighbors estimates (in particular, node $i$ will receive $x_j(k)$ from every neighbor $j\in N_i$).  We also require that nodes send the estimates of their neighbors, e.g.~node $i$ will receive together with the vector $\mathbf{W}_j^{(k-1)}$ another vector $\mathbf{X}_j(k-1)$ from every neighbor $j\in N_i$ where $\mathbf{X}_j(k-1)$ is the vector of the estimates of the neighbors of node $j$.
With such additional information, the following simple algorithm allows nodes to detect a misbehaving neighbor:
\begin{quote}
\underline{Misbehaving Neighbor Detection Algorithm - Node $i$}\\
$\{x_j(k),\mathbf{X}_j(k-1),\mathbf{W}_j^{(k-1)} \}$: the message received from a neighbor $j$ at iteration $k$\\
$\alpha(i)$:  index of a node $i$ in the corresponding vector\\
{\bf for all} $j\in N_i$\\
\hspace*{1cm} $C=w_{jj}(k-1)x_j(k-1)+\mathbf{X}_j^T(k-1)\mathbf{W}_j^{(k-1)}$\\
\hspace*{1cm} {\bf if} $\left(x_j(k)\neq C \right)$ or $\left(x_i(k-1)\neq \left(\mathbf{X}_j(k-1)\right)_{\alpha(i)}\right)$ or $\left(w_{ij}(k-1)\neq \left(\mathbf{W}_{j}^{(k-1)}\right)_{\alpha(i)}\right)$\\
\hspace*{1.5cm} Declare $j$ as misbehaving node.\\
\hspace*{1cm} {\bf end if}\\
{\bf end for}
\end{quote}
 
The first condition ($x_j(k)\neq w_{jj}(k-1)x_j(k-1)+ \mathbf{X}_j^T(k-1)\mathbf{W}_j^{(k-1)}$) corresponds to the definition of a misbehaving node and allows neighbors to detect a node sending a wrong estimate. The second and third conditions ($x_i(k-1)\neq \left(\mathbf{X}_j(k-1)\right)_{\alpha(i)}$) or ($w_{ij}(k-1)\neq \left(\mathbf{W}_{j}^{(k-1)}\right)_{\alpha(i)}$) detect if  node $j$ is modifying the content of any element in the vectors $\mathbf{X}_j(k-1)$ and $\mathbf{W}_j^{(k-1)}$ before sending them to its neighbors. More precisely, because of Assumption 2, if a node changes any element in the previously mentioned vectors, then this message will reach all neighbors including the neighbors concerned by this modification. These neighbors will remark this modification by checking the second and the third condition, and, due to Assumption~1, they will declare the node as misbehaving.

Once a node is declared a misbehaving node, the others can ignore it by simply assigning a null weight to its links in the following iterations.

\section{Conclusion}\label{Conc}

We have proposed in this report an approximated solution for the Fastest Distributed Linear Averaging (FDLA) problem by minimizing the Schatten $p$-norm of the weight matrix. Our approximated algorithm converges to the solution of the FDLA problem as $p$ approaches $\infty$, and in comparison to it has the advantage to be suitable for a distributed implementation. Moreover, simulations on random and real networks show that the algorithm outperforms other common distributed algorithms for weight selection.  

\bibliographystyle{./IEEEtran}
\bibliography{./IEEEabrv,./trace_bib}


\end{document}